\documentclass[12pt]{article}
\usepackage{mlmodern}
\usepackage{pdfrender,xcolor}
\usepackage{graphicx,color}
\usepackage{amsmath,amsfonts,amssymb,amsthm,MnSymbol}
\usepackage[hidelinks,colorlinks=true,linkcolor=blue,citecolor=blue]{hyperref}
\usepackage{enumerate}
\usepackage{algorithm}
\usepackage{algpseudocode}
\usepackage{rsfso}
\usepackage{bbm}
\usepackage{cancel}

\usepackage{graphicx}
\usepackage{subfig}
\usepackage{wrapfig}
\usepackage{tikz} %

\usepackage{natbib}
\usepackage{ulem}

\definecolor{darkgreen}{rgb}{0.06, 0.56, 0.2}

\newcommand{\E}{\mathbb{E}}

\newcommand{\R}{\mathbb{R}}
\newcommand{\Prob}{\mathbb{P}}
\newcommand{\Z}{\mathbb{Z}}

\def\cala{{\mathcal A}}

\def\calf{{\mathcal F}}

\DeclareMathOperator{\Var}{Var}

\newcommand{\pend}{\hfill \thicklines \framebox(6.6,6.6)[l]{}}

\newenvironment{proof*}[1]{\noindent {\sc  #1} \rm}{\pend}

\usepackage{mathtools}

\newtheorem{lemma}{Lemma}[section]
\newtheorem{assumption}{Assumption}[section]
\newtheorem{proposition}{Proposition}[section]

\newcommand{\setsection}[2] {
\setcounter{section}{#1}
\setcounter{subsection}{0}
\setcounter{equation}{0}
\setcounter{conjecture}{0}
\setcounter{assumption}{0}
\setcounter{question}{0}
\setcounter{definition}{0}
\setcounter{theorem}{0}
\setcounter{corollary}{0}
\setcounter{lemma}{0}
\setcounter{proposition}{0}
\setcounter{remark}{0}
\setcounter{appen}{0}
\setsection*{\large \bf \thesection. #2}}

\allowdisplaybreaks

\begin{document}
\title{\bf \Large  Asymptotic Product-form Steady-state for Multiclass Queueing Networks:
A Reentrant Line Case Study}

\author{J.G. Dai\\Cornell University\\ \and Dongyan (Lucy) Huo\\Cornell University\\}
\date{\today \medskip\\ }
\normalsize 


\maketitle

\begin{abstract}
This paper serves as a companion to ``Asymptotic Product-form Steady-state for Multiclass Queueing Networks with SBP Service Policies in Multi-scale Heavy Traffic''~\citep{DaiHuo2024}. In this short paper, we illustrate the main results of~\citet{DaiHuo2024} through a two-station, five-class reentrant line under a specific static buffer priority policy, while avoiding heavy notations.
For this example, we prove the asymptotic steady-state limit and uniform moment bound under general inter-arrival and service time distributions.

\end{abstract}

 \begin{quotation}

\medskip

\end{quotation}

\section{Introduction}
\label{sec:intro}

In the work by~\cite{DaiHuo2024}, the authors investigate the stationary distribution of the scaled queue length vector process in multi-class queueing networks operating under static buffer priority (SBP) service policies. They show that, under multi-scale heavy traffic, the stationary distribution of the scaled queue lengths converges to a product-form limit, with each component following an exponential distribution. Moreover, they prove uniform moment bounds for appropriately scaled low-priority queue lengths, provided the unscaled high-priority queue lengths have uniform moment bounds and a certain reflection matrix is a $\mathcal{P}$-matrix. This moment bound is crucial for proving the product-form limit.

This companion paper focuses on a specific example: a two-station, five-class reentrant line operating under a given SBP policy. Specializing in this concrete setting, we prove the asymptotic product-form limit in \cite{DaiHuo2024}, and argue that the assumptions for the uniform moment bound are satisfied and subsequently establish the uniform moment bound. As such, we demonstrate the key technical contributions while deliberately avoiding complex notation.

\textbf{Paper Organization:}
We first set up the two-station five-class re-entrant line and the service policy SBP in Section~\ref{sec:2s5c}. We next state our assumptions and present our main results in Section~\ref{sec:main-2s5c}, and subsequently provide the proofs in Section~\ref{sec:proof-2s5c} and~\ref{sec:mom-proof-2s5c}. We conclude the paper in Section~\ref{sec:conclude}. For numerical experiments for this reentrant line example, we refer readers back to the full paper~\citep{DaiHuo2024}.

\section{Network and Service Discipline}
\label{sec:2s5c}

Figure~\ref{fig:2s5c-fig} depicts the two-station five-class re-entrant line that we are going to discuss.
\begin{figure}[htbp]
    \centering
    \begin{tikzpicture}
    \draw (0, -0.5) rectangle (2, 2);
    \draw (4, -0.5) rectangle (6, 2);
    \draw[-] (-1, 1.5) -- (7,1.5);
    \draw[-] (-1, 0.75) -- (6.5,0.75);
    \draw[-] (7, 1.5) -- (7,-1.5);
    \draw[-] (6.5, 0.75) -- (6.5,-1);
    \draw[-] (6.5, -1) -- (-0.5,-1);
    \draw[-] (7, -1.5) -- (-1,-1.5);
    \draw[-] (-1, -1.5) -- (-1,0.75);
    \draw[-] (-0.5, -1) -- (-0.5,0);
    \draw[->] (-0.5, 0) -- (2.5,0);
    \draw[->] (-0.5, 0) -- (0,0);
    \draw[->] (-0.5, 1.5) -- (0,1.5);
    \draw[->] (-0.5, 0.75) -- (0,0.75);

     \draw[->] (3.5, 1.5) -- (4,1.5);
    \draw[->] (3.5, 0.75) -- (4,0.75);

    \draw (1, 1.7) node {$m_1$};
    \draw (1, 0.95) node {$m_3$};
    \draw (1, 0.2) node {$m_5$};

    \draw (5, 1.7) node {$m_2$};
    \draw (5, 0.95) node {$m_4$};

    \draw (1, 2.5) node {Station 1};
    \draw (5, 2.5) node {Station 2};
    
    \end{tikzpicture}
    \caption{Two-Station Five-Class Re-entrant Line}
    \label{fig:2s5c-fig}
\end{figure}
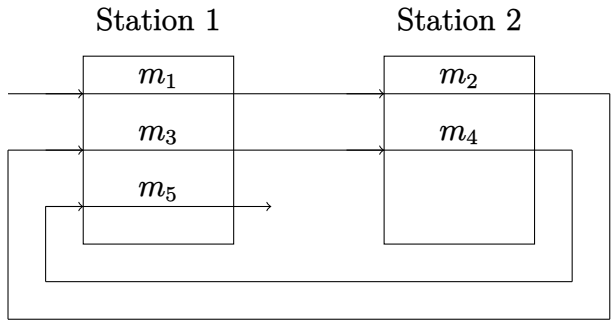
Each rectangle denotes a station with a single server, and each server processes jobs one at a time. Jobs arrive externally following a renewal process. The inter-arrival times $\{T_{e,1}(i)/\alpha_1, i\geq1\}$ are assumed to be independently identically distributed (i.i.d.) with mean $1/\alpha_1$ and $\E[T_{e,1}(1)]=1$.
Each job goes through five steps in the system before exiting, following the flow shown in the figure. Steps 1, 3, and 5 are at station 1, while steps 2 and 4 are at station 2. When a job finishes its processing in step $k(<5)$, it will move to buffer $k+1$ to wait for the start of its next step. We assume that each buffer has infinite capacity. 
Following~\cite{Harr1988}, we adopt the notion of job classes, and we denote a job to be of class $k$ if it is currently being processed in step $k$ or waiting in buffer $k$ for its turn. In the remainder of this paper, we will use the terms ``class'' and ``buffer'' interchangeably, as we consider that the job currently being processed in step $k$ is still in the buffer. For simplicity, we deviate from the convention in~\cite{DaiHuo2024} of reserving the index $j$ for the lowest-priority class at station $j$. Instead, we index the classes based on the traffic flow as shown in Figure~\ref{fig:2s5c-fig}. 
The processing times for each class are assumed to be i.i.d.\ with mean $m_k$, and we denote the sequence of processing times of class $k$ jobs as $\{m_kT_{s,k}(i),i\geq1\}$ with $\E[T_{s,k}(1)]=1$. We assume that the sequences
$\{T_{e,1}(i)/\alpha_1\}$ and $\{m_kT_{s,k}(i)\}$, $k\in\{1,\ldots,5\},$
are defined on the same probability space $(\Omega,\calf, \mathbb{P})$, and are mutually independent.
Furthermore, we place the following moment condition on $T_{e,1}$ and $T_{s,k}$ for $k=1,\ldots,5$.
\begin{assumption}
\label{assumption:2s5c-t-moment}
    There exists a small $\delta_0>0$ such that
\begin{equation*}
    \E[T_{e,1}^{3+\delta_0}]<\infty,\quad\text{and}\quad \E[T_{s,k}^{3+\delta_0}]<\infty,\quad k\in\{1,\ldots,5\}.
\end{equation*}
\end{assumption}

Subsequently, we denote
\begin{align}
    &c^2_{e,1}\equiv c^2(T_{e,1})=\Var(T_{e,1})/(\E[T_{e,1}])^2=\Var(T_{e,1}),\label{eq:ce-def}\\
    &c^2_{s,k}\equiv c^2(T_{s,k})=\Var(T_{s,k})/(\E[T_{s,k}])^2=\Var(T_{s,k}),\quad k=1,\ldots,5,\label{eq:cs-def}
\end{align}
where $c^2(U)$ denotes the squared coefficient of variation (SCV) of random variable $U$.

After completing the current job, the server must decide on the next job to work on, and this decision gives rise to a service policy. For this network, we assume the service discipline is non-idling and we study the following SBP service discipline,
\begin{equation}
\label{eq:2s5c-pri}
    \{(5,3,1),(2,4)\}.
\end{equation}
For jobs at station 1, we assign the highest priority to class $5$, the next priority to class $3$, and the lowest priority to class $1$. For jobs at station 2, we have class $2$ jobs of the highest priority and class $4$ jobs of the lowest. 
The server will always begin processing jobs from the non-empty class of the highest priority, and will only move on to jobs of lower priority classes, say class $1$, when no more jobs of strictly high priorities remain at the station, i.e., buffer $3$ and $5$ are empty.
Within each class, we assume that the jobs are served in a first-come first-serve manner. We additionally assume that the service discipline is preemptive-resume, that is, when a job of a higher priority than the one currently being served arrives at the station, the service of the current job will be interrupted. When all jobs of higher priorities at the station have been served, the previously interrupted job continues from where it has been left off.

\subsection{Markov Process and Stability}
Let $Z_k(t)$ denote the number of class $k$ jobs at time $t$, including possibly the one in service, for $k=1,\ldots,5$. Denote the queue length vector by
\begin{equation*}
    Z(t)=\Big(Z_1(t),Z_2(t),Z_3(t),Z_4(t),Z_5(t)\Big)',
\end{equation*}
where the superscript $'$ denotes the transpose operator.
When the inter-arrival and service times are all exponentially distributed, $\{Z(t),t\geq0\}$ then is a CTMC. However, when working with general inter-arrival and service times, $\{Z(t),t\geq0\}$ itself no longer forms a valid Markov process. To address this, we expand the state space to incorporate two auxiliary processes, $R_{e,1}(t)$, which denotes the remaining inter-arrival time of the next class $1$ job, and
\begin{equation*}
    R_s(t)=\Big(R_{s,1}(t),R_{s,2}(t),R_{s,3}(t),R_{s,4}(t),R_{s,5}(t)\Big)',
\end{equation*}
which denotes the remaining service time of the leading class $k$ job at time $t$, assuming that the server at the station devotes its entire service capacity to this job. If $Z_k(t)=0$, i.e., there is no class $k$ job in the buffer, $R_{s,k}(t)$ is then set to be the service time of the next class $k$ job in service.
Together,
\begin{equation*}
    X(t)=\big(Z'(t),R_{e,1}(t), R_s'(t)\big)',
\end{equation*}
is a valid Markov process on $\Z_+^5\times\R_+^6$. 
To lighten the notation, we omit the superscript $'$, and all vectors are defaulted to column vectors unless specified otherwise.

Let $\rho_j$, $j=1,2$ denote the traffic intensity at each station. The traffic intensities are computed as follows and are assumed to be less than 1. 
\begin{equation}
\label{eq:2s5c-rho}
    \rho_1=\alpha_1(m_1+m_3+m_5)<1,\quad\text{and}\quad
    \rho_2=\alpha_1(m_2+m_4)<1.
\end{equation}

We call a system ``stable,'' if the Markov process $\{X(t),t\geq0\}$ is positive Harris recurrent with a unique stationary distribution. Denote the random vector following the stationary distribution as
\begin{align}
    X&=(Z,R_{e,1},R_s),\label{eq:2s5c-state}\\
    \text{where } Z&=(Z_1,Z_2,Z_3,Z_4,Z_5)\quad\text{and}\quad
    R=(R_{s,1},R_{s,2},R_{s,3},R_{s,4},R_{s,5}).\nonumber
\end{align}
When the system is stable, we define the following quantities.
    \begin{align*}
    \beta_1&=\Prob(Z_1=0,Z_3=0, Z_5=0)=1-\alpha_1(m_1+m_3+m_5)=1-\rho_1,\quad \\
    \beta_3&=\Prob(Z_3=0, Z_5=0)=1-\alpha_1(m_3+m_5),\quad \beta_5=\Prob(Z_5=0)=1-\alpha_1m_5,\\
    \beta_4&=\Prob(Z_2=0, Z_4=0)=1-\alpha_1(m_2+m_4)=1-\rho_2,\quad \beta_2=\Prob(Z_2=0)=1-\alpha_1m_2,
\end{align*}
where $\beta_k$ denotes the excess capacity after serving all jobs at the station with priorities as high as $k$. Hence, $\rho_1,\rho_2<1$ is necessary for the two servers to remain not overloaded in the long run.

However, $\rho_1,\rho_2<1$ is not sufficient for the stability of this network operating under the specified SBP~\eqref{eq:2s5c-pri}, as discussed in \cite{DaiVand2000}. 
We need to additionally assume the following condition to ensure a stable system, that
\begin{equation}
\label{eq:2s5c-virtual-rho}
    \rho_v=\alpha_1(m_2+m_5)<1.
\end{equation}
When the system is stable, one may intuitively interpret the traffic intensity as the server utilization rate. 
The inequalities in~\eqref{eq:2s5c-rho} and \eqref{eq:2s5c-virtual-rho} together define the stability region for the specified SBP policy in \eqref{eq:2s5c-pri}.
For a general MCN, we would like to note that characterizing its stability region remains an open problem.

\section{Multi-scale Heavy-traffic and Product-form Limit}
\label{sec:main-2s5c}
To study the heavy-traffic asymptotic, we consider a sequence of the queueing networks indexed by $r\in(0,1)$, with $r$ tending towards $0$.
Under this set-up, the unitized inter-arrival and service times $T_{e,k}$ and $T_{s,k}$, routing, and the SBP policy specified in~\eqref{eq:2s5c-pri} are held constant throughout the sequence of networks indexed by $r$.
\begin{assumption}[Multi-scale heavy traffic]
\label{assumption:multi-scale-2s5c}
The external arrival rate $\alpha_1$ is assumed to be fixed and independent of $r$, and for simplicity, $\alpha_1=1$. Only the mean service time $m_k^{(r)}$ is assumed to vary and depend on $r$, 
\begin{equation*}
    m_k^{(r)}=\begin{cases}
        (1-r)m_k& k=1,3,5\\
        (1-r^2)m_k&k=2,4,
    \end{cases}
\end{equation*}
with $m_1+m_3+m_5=m_2+m_4=1$ and $m_2+m_5<1.$
\end{assumption}

Thus, under Assumption~\ref{assumption:multi-scale-2s5c}, for $r\in(0,1)$,
\begin{equation*}
    r^{-1}(1-\rho_1^{(r)})=1,\quad\text{and}\quad r^{-2}(1-\rho_2^{(r)})=1,
\end{equation*}
which implies that the traffic intensities at station 1 and 2 approach the critical load $1$ at segregated orders of $r$.
This segregation of the degree of $r$ corresponds to the rationale behind terming it the ``multi-scale heavy traffic condition.''

Under this setup, for $r\in(0,1)$, $m_2^{(r)}+m_5^{(r)}<1.$
Therefore, $\{X^{(r)}(t), t\geq0\}$ is positive Harris recurrent, and we denote $X^{(r)}=(Z^{(r)}, R_{e,1}, R_s^{(r)})$ as the state vector under the stationary distribution.

\begin{proposition}
\label{prop:2s5c-result}
Under Assumption~\ref{assumption:2s5c-t-moment} and~\ref{assumption:multi-scale-2s5c},
there exist a random vector $(Z_1^\ast, Z_4^\ast)\in\R_+^2$ such that
    \begin{equation*}
        (rZ_1^{(r)}, rZ_2^{(r)}, rZ_3^{(r)}, r^2Z_4^{(r)}, rZ_5^{(r)})\Rightarrow (Z_1^\ast, 0, 0, Z_4^\ast,0),\quad\text{as }r\to0,
    \end{equation*}
    where ``$\Rightarrow$'' denotes convergence in distribution.
    
   Moreover, $Z_1^\ast$ and $Z_4^\ast$ are independent and have the following marginal distributions
   \begin{equation*}
        Z_1^\ast\sim\exp\big(1/d_1\big),\quad\text{and}\quad Z_4^\ast\sim\exp\big(1/d_4\big),
    \end{equation*}
    with
    \begin{align}
        d_1&=\frac{\alpha_1}{2(m_1+m_3-m_5\frac{m_2}{m_4})}\Big((m_1+m_3-m_5\frac{m_2}{m_4})^2c_{e,1}^2+m_1^2c_{s,1}^2+m_3^2c_{s,3}^2+m_5^2c_{s,5}^2\nonumber\\
        &\qquad\qquad\qquad\qquad\qquad\quad+(\frac{m_5}{m_4})^2\big(m_2^2c_{s,2}^2+m_4^2c_{s,4}^2\big)\Big)\label{eq:2s5c-d1}\\
        d_4&=\frac{\alpha_1}{2m_4}\Big((m_2+m_4)^2c_{e,1}^2+m_2^2c_{s,2}^2+m_4^2c_{s,4}^2\Big).\label{eq:2s5c-d4}
    \end{align}
\end{proposition}

\section{Proof of Product-form Limit}
\label{sec:proof-2s5c}
In this section, we prove Proposition~\ref{prop:2s5c-result}. Despite that the network assumes such a simple structure, there have not been any results providing explicit formulas for performance evaluation. Thus, our result is nontrivial even in this most fundamental setup. 
We start by outlining the key components that lay the groundwork for the proof.

\noindent\textbf{Moment Generating Functions.}
The significance of moment generating functions (MGFs) comes to the forefront, as the convergence of MGF implies the desired distributional convergence.

Define $g_\theta(z)=\exp\big(\langle\theta, z\rangle\big)$ with $\theta\in\R_-^5$ and $\phi^{(r)}(\theta)=\E[g_\theta(Z^{(r)})]$. We observe that
\begin{equation*}
    \phi^{(r)}(r\eta_1,r\eta_2,r\eta_3,r^2\eta_4,r\eta_5)=\E\Big[\exp\Big(\eta_1rZ_1^{(r)}+\eta_2rZ_2^{(r)}+\eta_3rZ_3^{(r)}+\eta_4r^2Z_4^{(r)}+\eta_5rZ_5^{(r)}\Big)\Big],
\end{equation*}
which is the MGF of the random vector $(rZ_1^{(r)}, rZ_2^{(r)}, rZ_3^{(r)},r^2Z_4^{(r)},rZ_5^{(r)})$ of interest.
Hence, proving Proposition~\ref{prop:2s5c-result} is equivalent to show that
\begin{equation*}
    \lim_{r\to0}\phi^{(r)}(r\eta_1,r\eta_2,r\eta_3,r^2\eta_4,r\eta_5)=\frac{1}{1-d_1\eta_1}\frac{1}{1-d_4\eta_4},
\end{equation*}
with $d_1$ and $d_4$ defined in \eqref{eq:2s5c-d1} and \eqref{eq:2s5c-d4}.
The product form implies the independence of $Z_1^\ast$ and $Z_4^\ast$, and each conforms to the MGF of an exponential distribution.

\noindent\textbf{Uniform Moment Bounds and State Space Collapse.}    
In Theorem 3.7 and Section 4.1 of \cite{CaoDaiZhan2022}, it has been established that, under Assumption~\ref{assumption:2s5c-t-moment}, there exists $r_0\in(0,1)$ and $\epsilon_0>0$ such that
\begin{equation}
    \label{eq:2s5c-ssc}
    \sup_{r\in(0,r_0)}\E\Big[(Z_2^{(r)}+Z_3^{(r)}+Z_5^{(r)})^{2+\epsilon_0}\Big]<\infty.
\end{equation}
This uniform moment bound on high-priority queue lengths translates into the following state space collapse (SSC) result in MGFs. 
\begin{lemma}
\label{lem:2s5c-high-ssc}
    For fixed $\eta\in\R_-^5$, we have
    \begin{align*}
    &\lim_{r\to0}\phi^{(r)}(r\eta_1,0,0,r^2\eta_4,0)-\phi^{(r)}(r\eta_2,r\eta_2,r\eta_3,r^2\eta_4,r\eta_5)=0.\\
    &\lim_{r\to0}\phi_k^{(r)}(r\eta_1,0,0,r^2\eta_4,0)-\phi_k^{(r)}(r\eta_2,r\eta_2,r\eta_3,r^2\eta_4,r\eta_5)=0,\quad k\in\{1,\ldots,5\}.
\end{align*}
\end{lemma}
$\phi_k^{(r)}$ denotes the conditional MGF and is defined as follows,
\begin{align*}
    &\phi^{(r)}_1(\theta)=\E[g_\theta(Z^{(r)})|Z^{(r)}_1=0, Z^{(r)}_3=0, Z^{(r)}_5=0],\\
    &\phi^{(r)}_3(\theta)=\E[g_\theta(Z^{(r)})| Z^{(r)}_3=0, Z^{(r)}_5=0],\quad \phi^{(r)}_5(\theta)=\E[g_\theta(Z^{(r)})| Z^{(r)}_5=0],\\
    &\phi_4^{(r)}(\theta)=\E[g_\theta(Z^{(r)})| Z_2^{(r)}=0, Z_4^{(r)}=0],\quad\phi^{(r)}_2(\theta)=\E[g_\theta(Z^{(r)})| Z_2^{(r)}=0].
\end{align*}

\begin{proof}
    Given $\eta\in\R_-^5$, we first recall that
    \begin{equation*}
        \phi^{(r)}(r\eta_1,r\eta_2,r\eta_3,r^2\eta_4,r\eta_5)=\E\Big[\exp\Big(\eta_1rZ_1^{(r)}+\eta_2rZ_2^{(r)}+\eta_3rZ_3^{(r)}+\eta_4r^2Z_4^{(r)}+\eta_5rZ_5^{(r)}\Big)\Big].
    \end{equation*}
    Hence, we have
    \begin{align*}
        &\phi^{(r)}(r\eta_1,0,0,r^2\eta_4,0) - \phi^{(r)}(r\eta_1,r\eta_2,r\eta_3,r^2\eta_4,r\eta_5)\\
        &=\E\Big[\exp\Big(\eta_1rZ_1^{(r)}+\eta_4r^2Z_4^{(r)}\Big)\Big(1-\exp\Big(\eta_2rZ_2^{(r)}+\eta_3rZ_3^{(r)}+\eta_5rZ_5^{(r)}\Big)\Big)\Big]\\
        &\leq \E\Big[1-\exp\Big(\eta_2rZ_2^{(r)}+\eta_3rZ_3^{(r)}+\eta_5rZ_5^{(r)}\Big)\Big]
        \leq |\eta_2|r\E[Z_2^{(r)}]+|\eta_3|r\E[Z_3^{(r)}]+|\eta_5|r\E[Z_5^{(r)}].
    \end{align*}
    Recall the uniform moment bound on the unscaled high-priority queue lengths shown in~\eqref{eq:2s5c-ssc}, we can show the desired SSC, as $r\to0$,
    \begin{equation*}
        \phi^{(r)}(r\eta_1,0,0,r^2\eta_4,0) - \phi^{(r)}(r\eta_1,r\eta_2,r\eta_3,r^2\eta_4,r\eta_5)\to0.
    \end{equation*}

    We prove the following conditional version of SSC, and the remaining conditional SSCs would follow a similar proof strategy.
    \begin{align*}
        &\phi_4^{(r)}(r\eta_1,0,0,r^2\eta_4,0) - \phi_4^{(r)}(r\eta_1,r\eta_2,r\eta_3,r^2\eta_4,r\eta_5)\\
        &=\E\Big[\exp\Big(\eta_1rZ_1^{(r)}+\eta_4r^2Z_4^{(r)}\Big)\Big(1-\exp\Big(\eta_2rZ_2^{(r)}+\eta_3rZ_3^{(r)}+\eta_5rZ_5^{(r)}\Big)\Big)\mid Z_2^{(r)}=0, Z_4^{(r)}=0\Big]\\
        &\leq \E\Big[1-\exp\Big(\eta_2rZ_2^{(r)}+\eta_3rZ_3^{(r)}+\eta_5rZ_5^{(r)}\Big)\mid Z_2^{(r)}=0, Z_4^{(r)}=0\Big]\\
        &\leq \frac{1}{\Prob(Z_2^{(r)}=0, Z_4^{(r)}=0)} \E\Big[r(|\eta_3|Z_3^{(r)}+|\eta_5|Z_5^{(r)})\mathbbm1\{Z_2^{(r)}=0, Z_4^{(r)}=0\}\Big].
    \end{align*}
    Next, we use H\"{o}lder's inequality and further obtain
    \begin{align*}
        &\E\Big[(|\eta_3|Z_3^{(r)}+|\eta_5|Z_5^{(r)})\mathbbm1\{Z_2^{(r)}=0, Z_4^{(r)}=0\}\Big]\\\
        &\leq \E\Big[(|\eta_3|Z_3^{(r)}+|\eta_5|Z_5^{(r)})^p\Big]^{1/p}\E\Big[\mathbbm1\{Z_2^{(r)}=0, Z_4^{(r)}=0\}\Big]^{1/q}.
    \end{align*}
    Consider setting $p=2+\epsilon_0$, then as $r\to0$, we have
    \begin{align*}
        \phi_4^{(r)}(r\eta_1,0,0,r^2\eta_4,0) - &\phi_4^{(r)}(r\eta_1,r\eta_2,r\eta_3,r^2\eta_4,r\eta_5)
        \\&\leq r^{\epsilon_0/(2+\epsilon_0)}\E\Big[(|\eta_3|Z_3^{(r)}+|\eta_5|Z_5^{(r)})^{2+\epsilon_0}\Big]^{1/(2+\epsilon_0)}\to0.
    \end{align*}
    As such, we have proven the desired conditional SSC.
\end{proof}

The vanishing of high-priority jobs in the limit is not surprising and a direct consequence of Lemma~\ref{lem:2s5c-high-ssc}.
Hence, proving Proposition~\ref{prop:2s5c-result} is equivalent to showing the following limit,
\begin{equation}
\label{eq:2s5c-limit-ssc}
    \lim_{r\to0}\phi^{(r)}(r\eta_1,0,0,r^2\eta_4,0)=\frac{1}{1-d_1\eta_1}\frac{1}{1-d_4\eta_4}.
\end{equation}

Similar to the uniform moment bound for high-priority queue length vectors in~\eqref{eq:2s5c-ssc}, the low-priority queue length vectors, when appropriately scaled, also have uniformly bounded moments, as stated in the following proposition, with proof in Section~\ref{sec:mom-proof-2s5c}.
\begin{proposition}
\label{prop:2s5c-moment-low}
    Under Assumption~\ref{assumption:2s5c-t-moment} and~\ref{assumption:multi-scale-2s5c}, there exists $r_0\in(0,1)$ and $\epsilon_0>0$ such that
\begin{equation*}
    \sup_{r\in(0,r_0)}\E[(rZ_1^{(r)})^{2+\epsilon_0}]<\infty\quad\text{and}\quad\sup_{r\in(0,r_0)}\E[(r^2Z_4^{(r)})^{2+\epsilon_0}]<\infty.
\end{equation*}
\end{proposition}

The uniform moment bound for the scaled low-priority queue length vectors leads to the following SSC results in MGFs. 
We later shall see the significance of this SSC result in establishing the product-form limit~\eqref{eq:2s5c-limit-ssc}.

\begin{proposition}
\label{prop:2s5c-ssc-low}
For any fixed $\eta\in\R_-^5$,
    \begin{align*}
        &\lim_{r\to0}\phi^{(r)}(0,0,0,r^2\eta_4,0)-\phi^{(r)}(r^2\eta_1,0,0,r^2\eta_4,0)=0,\\
        &\lim_{r\to0}\phi_k^{(r)}(0,0,0,r^2\eta_4,0)-\phi_k^{(r)}(r^2\eta_1,0,0,r^2\eta_4,0)=0,\quad k\in\{1,\ldots,5\}.
    \end{align*}
\end{proposition}

\begin{proof}
Given $\eta\in\R_-^5$, we note that
    \begin{align*}
        \phi^{(r)}(0,0,0,r^2\eta_4,0)-\phi^{(r)}(r^2\eta_1,0,0,r^2\eta_4,0)
        &\leq \E[1-\exp(\eta_1r^2Z_1^{(r)})]
        \leq r\eta_1\E[rZ_1^{(r)}].
    \end{align*}
    By Proposition~\ref{prop:2s5c-moment-low}, $\sup_{r\in(0,r_0)}\E[rZ_1^{(r)}]<\infty$ and hence the last term converges to 0.
    
    For the proof of conditional version, we demonstrate the proof of one case, and the remaining conditional SSC follows similarly.
    \begin{align*}
        &\phi_4^{(r)}(0,0,0,r^2\eta_4,0)-\phi_4^{(r)}(r^2\eta_1,0,0,r^2\eta_4,0)
        \leq \E\Big[1-\exp(\eta_1r^2Z_1^{(r)})\mid Z_2^{(r)}=0, Z_4^{(r)}=0\Big]\\
        &\leq \frac{|\eta_1|}{\Prob(Z_2^{(r)}=0,Z_4^{(r)}=0)}\E\Big[r^2Z_1^{(r)}\mathbbm1\{Z_2^{(r)}=0,Z_4^{(r)}=0\}\Big]
        = |\eta_1|\E[Z_1^{(r)}\mathbbm1\{Z_2^{(r)}=0,Z_4^{(r)}=0\}]
    \end{align*}
    We apply H\"{o}lder's inequality with $p=2+\epsilon_0$, and obtain that as $r\to0$.
    \begin{equation*}
        \phi_4^{(r)}(0,0,0,r^2\eta_4,0)-\phi_4^{(r)}(r^2\eta_1,0,0,r^2\eta_4,0)
        \leq r^{\epsilon_0/(2+\epsilon_0)}\E\Big[(rZ_1^{(r)})^{2+\epsilon_0}\Big]^{1/(2+\epsilon_0)}\to0.
    \end{equation*}
    As such, we have proven the desired conditional SSC.
\end{proof}

\noindent\textbf{(Asymptotic) Basic Adjoint Relationship.}
To prove the desired steady-state limit, we adopt the basic adjoint relationship (BAR) technique, which is an alternative to the ``interchange of limits'' approach that has dominated the literature for the past two decades.
The BAR approach directly analyzes the stationary distribution of the Markov process. 
The BAR for MCNs with general inter-arrival and service time distributions was fully derived in \cite{BravDaiMiya2023}. To avoid repetition, we omit the detailed derivation here and direct interested readers to~\cite{BravDaiMiya2023} for further details. 

For ease of comprehension, we first illustrate the proof for exponential interarrival and service time distributions in Section~\ref{sec:exp-proof}, and explain how to extend the proof for general distributions in Section~\ref{sec:general-proof}.

\subsection{Exponential Distribution}
\label{sec:exp-proof}

When the interarrival and service times are all exponentially distributed, $\{Z(t),t\geq0\}$ forms a valid continuous time Markov chain (CTMC). The BAR characterizing the stationary distribution of the CTMC is
\begin{equation}
\label{eq:bar-exp}
    \E\Big[Gf(Z)\Big]=0\quad\text{for bounded functions } f:\Z_+^5\to\R,
\end{equation}
where for each $z\in\Z_+^5$ and each bounded $f$, the generator is defined as
\begin{align*}
    Gf(z)&=\alpha_1 (f(z+e^{(1)})-f(z))\\
    &+\mu_1(f(z+e^{(2)}-e^{(1)})-f(z))\mathbbm1\{z_1>0, z_3=z_5=0\}\\
    &+\mu_2(f(z+e^{(3)}-e^{(2)})-f(z))\mathbbm1\{z_2>0\}\\
    &+\mu_3(f(z+e^{(4)}-e^{(3)})-f(z))\mathbbm1\{z_3>0, z_5=0\}\\
    &+\mu_4(f(z+e^{(5)}-e^{(4)})-f(z))\mathbbm1\{z_4>0,z_2=0\}\\
    &+\mu_5(f(z-e^{(5)})-f(z))\mathbbm1\{z_5>0\}.
\end{align*}

Substituting the following exponential test function $f$ into BAR~\eqref{eq:bar-exp},
\begin{equation}
    f_\theta(z)=\exp\Big(\sum_{k=1}^5\theta_kz_k\Big),\quad\text{with } \theta=(\theta_1,\ldots,\theta_5)\in\R_-^5,
\end{equation}
we have the following relationship among MGFs, as established in~\cite{BravDaiMiya2023}.

\begin{lemma}[\cite{BravDaiMiya2023} Lemma 2.2]
\begin{align*}
        &\alpha_1\big(\gamma_1(\theta_1)+\sum_{k=1}^5\zeta_k(\theta)\big)\phi(\theta)\\
    &+\beta_1\mu_1\zeta_1(\theta)\Big(\phi(\theta)-\phi_1(\theta)\Big)+\beta_4\mu_4\zeta_4(\theta)\Big(\phi(\theta)-\phi_4(\theta)\Big)\\
    &+\beta_2\big(\mu_2\zeta_2(\theta)-\mu_4\zeta_4(\theta)\big)\big(\phi(\theta)-\phi_2(\theta)\big)\\
    &+\beta_3\big(\mu_3\zeta_3(\theta)-\mu_1^{(r)}\zeta(\theta)\big)\big(\phi(\theta)-\phi_3(\theta)\big)\\
    &+\beta_5\big(\mu_5\zeta_5(\theta)-\mu_3\zeta_3(\theta)\big)\big(\phi(\theta)-\phi_5(\theta)\big)=0,
\end{align*}
where
\begin{align*}
    \gamma_1(\theta_1)=e^{\theta_1}-1,\quad\text{and}\quad\zeta_k(\theta)=e^{\theta_{k+1}\mathbbm1\{k<5\}-\theta_k}-1,\quad k=1,\ldots,5.
\end{align*}
\end{lemma}

Next, When $\theta\equiv\theta(r)$, a function of $r$, such that $|\theta|=\sum_{k=1}^5|\theta_k|\leq cr$, it is easy to see that $\gamma$ and $\zeta$ admits the following Taylor expansions as $r\to0$, that
\begin{equation}
    \label{eq:2s5c-taylor-exp}
    \gamma_1(\theta_1)=\bar\gamma_1(\theta_1)+\tilde\gamma_1(\theta_1)+o(\theta_1^2)\quad\text{and}\quad
    \zeta_k(\theta)=\bar\zeta_k(\theta_1)+\tilde\zeta_k(\theta_1)+o(|\theta|^2),
\end{equation}
where
\begin{align*}
    &\bar\gamma_1(\theta_1)=\theta_1,\quad\tilde\gamma_1(\theta_1)=\frac{1}{2}\theta_1^2,\\
    &\bar\zeta_k(\theta)=\theta_{k+1}\mathbbm1\{k<5\}-\theta_k,\quad\tilde\zeta_k(\theta)=\frac{1}{2}(\theta_{k+1}\mathbbm1\{k<5\}-\theta_k)^2,\quad k\in\{1,\ldots,5\},\\
    &\gamma_1^\ast(\theta_1)=\bar\gamma_1(\theta_1)+\tilde\gamma_1(\theta_1),\quad\zeta_k^\ast(\theta)=\bar\zeta_k(\theta)+\tilde\zeta_k(\theta),\quad k\in\{1,\ldots,5\}.
\end{align*}

Bringing this $\theta(r)$ into the MGFs of the sequence of networks indexed with $r$, we obtain the following asymptotic BAR.
\begin{lemma}[\cite{BravDaiMiya2023} Lemma 2.3]
\begin{align*}
        &\alpha_1\big(\gamma_1^\ast(\theta_1)+\sum_{k=1}^5\zeta^\ast_k(\theta)\big)\phi^{(r)}(\theta)\\
    &+\beta_1^{(r)}\mu_1^{(r)}\zeta_1^\ast(\theta)\Big(\phi^{(r)}(\theta)-\phi_1^{(r)}(\theta)\Big)+\beta_1^{(r)}\mu_4^{(r)}\zeta_4^\ast(\theta)\Big(\phi^{(r)}(\theta)-\phi_4^{(r)}(\theta)\Big)\\
    &+\beta_2^{(r)}\big(\mu_2^{(r)}\zeta_2^\ast(\theta)-\mu_4^{(r)}\zeta_4^\ast(\theta)\big)\big(\phi^{(r)}(\theta)-\phi_2^{(r)}(\theta)\big)\\
    &+\beta_3^{(r)}\big(\mu_3^{(r)}\zeta_3^\ast(\theta)-\mu_1^{(r)}\zeta^\ast_1(\theta)\big)\big(\phi^{(r)}(\theta)-\phi_3^{(r)}(\theta)\big)\\
    &+\beta_5^{(r)}\big(\mu_5^{(r)}\zeta_5^\ast(\theta)-\mu_3^{(r)}\zeta_3^\ast(\theta)\big)\big(\phi^{(r)}(\theta)-\phi_5^{(r)}(\theta)\big)=o(|\theta|^2).
\end{align*}
\end{lemma}

Next, we note the following SSC result.
\begin{lemma}[\cite{BravDaiMiya2023} Lemma 2.6]
    \begin{equation*}
        \lim_{r\to0}\phi^{(r)}(\theta)-\phi_k^{(r)}(\theta)=0,\quad k\in\{2,3,5\}.
    \end{equation*}
\end{lemma}

Brining this SSC result into the asymptotic BAR, we can further simplify the asymptotic BAR to the following.
\begin{align*}
    &\alpha_1\big(\tilde\gamma_1(\theta_1)+\sum_{k=1}^5\tilde\zeta_k(\theta)\big)\phi^{(r)}(\theta)\\
    &+\beta_1^{(r)}\mu_1^{(r)}\bar\zeta_1(\theta)\Big(\phi^{(r)}(\theta)-\phi_1^{(r)}(\theta)\Big)+\beta_1^{(r)}\mu_4^{(r)}\bar\zeta_4(\theta)\Big(\phi^{(r)}(\theta)-\phi_4^{(r)}(\theta)\Big)\\
    &+\beta_2^{(r)}\big(\mu_2^{(r)}\bar\zeta_2(\theta)-\mu_4^{(r)}\bar\zeta_4(\theta)\big)\big(\phi^{(r)}(\theta)-\phi_2^{(r)}(\theta)\big)\\
    &+\beta_3^{(r)}\big(\mu_3^{(r)}\bar\zeta_3(\theta)-\mu_1^{(r)}\bar\zeta_1(\theta)\big)\big(\phi^{(r)}(\theta)-\phi_3^{(r)}(\theta)\big)\\
    &+\beta_5^{(r)}\big(\mu_5^{(r)}\bar\zeta_5(\theta)-\mu_3^{(r)}\bar\zeta_3(\theta)\big)\big(\phi^{(r)}(\theta)-\phi_5^{(r)}(\theta)\big)=o(|\theta|^2).
\end{align*}

Established upon the asymptotic BAR and SSC results, we now are ready to prove Proposition~\ref{prop:2s5c-result} under exponential interarrival and service times. 
\begin{proof}
The proof of the desired product-form limit follows an inductive argument, with the following three equations as key intermediary steps: for a given $\eta\in\R_-^5$, we have
\begin{align}
    & \lim_{r\to0}\phi^{(r)}(r\eta_1,0,0,r^2\eta_4,0)-\frac{1}{1-d_1\eta_1}\phi_1^{(r)}(0,0,0,r^2\eta_4,0)=0\label{eq:2s5c-step-1}\\
    & \lim_{r\to0}\phi^{(r)}(0,0,0,r^2\eta_4,0)-\phi_1^{(r)}(0,0,0,r^2\eta_4,0)=0\label{eq:2s5c-step-2}\\
    &\lim_{r\to0}\phi^{(r)}(0,0,0,r^2\eta_4,0)-\frac{1}{1-d_4\eta_4}=0\label{eq:2s5c-step-3}.
\end{align}
Taking these equations as given, we proceed inductively:
\begin{align*}
    \lim_{r\to0}\phi^{(r)}(r\eta_1,r\eta_2,r\eta_3,r^2\eta_4,r\eta_5)&=\lim_{r\to0}\phi^{(r)}(r\eta_1,0,0,r^2\eta_4,0) &&\text{Lemma~\ref{lem:2s5c-high-ssc}}\\
    &=\frac{1}{1-d_1\eta_1}\lim_{r\to0}\phi_1^{(r)}(0,0,0,r^2\eta_4,0)&&\eqref{eq:2s5c-step-1}\\
    &=\frac{1}{1-d_1\eta_1}\lim_{r\to0}\phi^{(r)}(0,0,0,r^2\eta_4,0)&&\eqref{eq:2s5c-step-2}\\
    &=\frac{1}{1-d_1\eta_1}\frac{1}{1-d_4\eta_4}.&&\eqref{eq:2s5c-step-3}
\end{align*}
As such, we have proven the desired limit in~\eqref{eq:2s5c-limit-ssc}.

To complete the proof, what remains is to prove~\eqref{eq:2s5c-step-1}--\eqref{eq:2s5c-step-3}. The proofs of the three equations are centered around carefully choosing different $\theta$ as functions of fixed $(\eta_1,\eta_4)\in\R_-^2$ and $r\in(0,1)$.

Starting with~\eqref{eq:2s5c-step-1}, we first choose the following $\theta$,
\begin{equation}
\label{eq:2s5c-theta-example}
   \begin{aligned}
        \theta(\eta,r)=\Big(r\eta_1+\frac{1}{m_2}r^2\eta_4&,\frac{m_3-m_5\frac{m_2}{m_4}}{m_1+m_3-m_5\frac{m_2}{m_4}}r\eta_1+\frac{1}{m_2}r^2\eta_2,
        \frac{m_3}{m_3-m_5\frac{m_2}{m_4}}r\eta_1+r^2\eta_4,\\
        r^2\eta_4&,\frac{m_5}{m_3-m_5\frac{m_2}{m_4}}r\eta_1+r^2\eta_4\Big).
    \end{aligned}
\end{equation}

For simplicity, in this proof, we assume that $m_3-m_5\frac{m_2}{m_4}\geq0$, though this assumption is not required. When $m_3-m_5\frac{m_2}{m_4}<0$, we can employ the truncation technique to address this issue, as discussed in detail in~\cite{BravDaiMiya2023} and~\cite{DaiHuo2024}.

    One key aspect of choosing $\theta$ strategically is to establish the following relationships,
    \begin{equation*}
        \mu_3^{(r)}\bar\zeta_3(\theta)-\mu_1^{(r)}\bar\zeta_1(\theta)=
        \mu_5^{(r)}\bar\zeta_5(\theta)-\mu_3^{(r)}\bar\zeta_3(\theta)=
        \mu_2^{(r)}\bar\zeta_2(\theta)-\mu_4^{(r)}\bar\zeta_4(\theta)=0,
    \end{equation*}
    which are satisfied under our choice of $\theta$ in~\eqref{eq:2s5c-theta-example}. 
    Substituting this $\theta$ into the asymptotic BAR~\eqref{eq:2s5c-asymptotic-bar} and making use of the multi-scale heavy traffic condition in Assumption~\ref{assumption:multi-scale-2s5c}, i.e., $\beta_1^{(r)}=r$ and $\beta_4^{(r)}=r^2$, we obtain
\begin{align*}
    &\alpha_1\big(\tilde\gamma_1(\theta_1)+\sum_{k=1}^5\tilde\zeta_k(\theta)\big)\phi(\theta)\\
    &+r\mu_1^{(r)}\Big(r\eta_1+\frac{1}{m_4}r^2\eta_4\Big)\Big(\phi(\theta)-\phi_1(\theta)\Big)
    +r^2\mu_4^{(r)}\Big(r^2\eta_4\Big)\Big(\phi(\theta)-\phi_4(\theta)\Big)=o(r^2),
\end{align*}
and we note that $\alpha_1\big(\tilde\gamma_1(\theta_1)+\sum_{k=1}^5\tilde\zeta_k(\theta)\big)=O(r^2)$ for $|\theta|=O(r)$ by construction in~\eqref{eq:2s5c-theta-example}.
Thus, we divide $r^2$ on both sides and as $r\to0$, we have
\begin{align*}
    &\lim_{r\to0}\frac{1}{r^2}\Big(\alpha_1\big(\tilde\gamma_1(\theta_1)+\sum_{k=1}^5\tilde\zeta_k(\theta)\big)\phi(\theta)\\
    &\qquad\qquad
    +r\mu_1^{(r)}\Big(r\eta_1+\frac{1}{m_4}r^2\eta_4\Big)\Big(\phi(\theta)-\phi_1(\theta)\Big)
    +r^2\mu_4^{(r)}\Big(r^2\eta_4\Big)\Big(\phi(\theta)-\phi_4(\theta)\Big)\Big)\\
    &=\lim_{r\to0}\frac{1}{r^2}\Big(\alpha_1\big(\tilde\gamma_1(\theta_1)+\sum_{k=1}^5\tilde\zeta_k(\theta)\big)\Big)\phi(\theta) + \mu_1^{(r)}\eta_1\Big(\phi(\theta)-\phi_1(\theta)\Big)
    =0.
\end{align*}
Simplifying the above equation, we arrive at the following equation, with $d_1$ defined in~\ref{eq:2s5c-d1},
\begin{equation}
\label{eq:2s5c-ssc-step-1}
    \lim_{r\to0}\phi^{(r)}(\theta)-\frac{1}{1-d_1\eta_1}\phi_1^{(r)}(\theta)=0.
\end{equation}

By the SSC results in Lemma~\ref{lem:2s5c-high-ssc} and Proposition~\ref{prop:2s5c-ssc-low}, we have
\begin{equation*}
    \lim_{r\to0}\phi^{(r)}(r\eta_1,0,0,r^2\eta_4,0)-\phi^{(r)}(\theta)=0\quad\text{and}\quad
    \lim_{r\to0}\phi_1^{(r)}(0,0,0,r^2\eta_4,0)-\phi_1^{(r)}(\theta)=0.
\end{equation*}
Therefore, substituting the above SSC results back into~\eqref{eq:2s5c-ssc-step-1}, we have proved~\eqref{eq:2s5c-step-1},
\begin{equation*}
    \lim_{r\to0}\phi^{(r)}(r\eta_1,0,0,r^2\eta_4,0)-\frac{1}{1-d_1\eta_1}\phi_1^{(r)}(0,0,0,r^2\eta_4,0)=0.
\end{equation*}

It is easy to prove~\eqref{eq:2s5c-step-2}, simply by setting $\eta_1=0$ in~\eqref{eq:2s5c-step-1}, and hence, we have
\begin{equation*}
    \lim_{r\to0}\phi^{(r)}(0,0,0,r^2\eta_4,0)-\phi_1^{(r)}(0,0,0,r^2\eta_4,0)=0.
\end{equation*}

Lastly, to prove~\eqref{eq:2s5c-step-3}, we consider the following choice of $\theta$,
\begin{equation*}
        \theta(\eta,r)=\Big(\frac{1}{m_4}r^2\eta_4, \frac{1}{m_4}r^2\eta_4, r^2\eta_4, r^2\eta_4, 0\Big).
    \end{equation*}
    Given this choice, we see that 
    \begin{equation*}
        \bar\zeta_1(\theta)=\bar\zeta_3(\theta)=\bar\zeta_5(\theta)=
        \mu_2^{(r)}\bar\zeta_2(\theta)-\mu_4^{(r)}\bar\zeta_4(\theta)=0.
    \end{equation*}
    Substituting this choice of $\theta$ into the asymptotic BAR, and following a similar simplification strategy as the previous steps, we obtain 
    \begin{equation*}
        \lim_{r\to0}\phi^{(r)}(\theta)-\frac{1}{1-d_4\eta_4}\phi_2^{(r)}(\theta)=0.
    \end{equation*}
Together with SSC results, we have established~\eqref{eq:2s5c-step-3}
\begin{equation*}
    \lim_{r\to0}\phi^{(r)}(0,0,0,r^2\eta_4,0)-\frac{1}{1-d_4\eta_4}=0.
\end{equation*}

As such, we have completed the proof of Proposition~\ref{prop:2s5c-result}. 
\end{proof}

\subsection{General Distribution}
\label{sec:general-proof}

Now we extend the above analysis to general interarrival and service time distributions, with a particular focus on bounded distributions. For discussions on handling unbounded distributions, we refer readers to~\cite{BravDaiMiya2023} and~\cite{DaiHuo2024}.
We shall see that once we have proved the exponential case, handling general distributions is not that much different.

With interarrival and service times now generally distributed, we turn our attention to the Markov process $\{X(t),t\geq0\}$ as defined in~\eqref{eq:2s5c-state}. The BAR characterizing the stationary distribution of the Markov process is provided in the following lemma.

\begin{lemma}[\cite{BravDaiMiya2023} Lemma 2.13]
    Let the random vector $X$ in~\eqref{eq:2s5c-state} denote the pre-jump state and $\Delta$ represent the increments corresponding to different jump events, i.e., external job arrival or service completion,
\begin{align*}
    &\Delta_{e,1}\equiv\Big(e^{(1)},e^{(1)}T_{e,1}/\alpha_1,0\Big)\\
    \text{and}\quad&\Delta_{s,k}\equiv\Big(-e^{(k)}+e^{(k+1)}\mathbbm1\{k<5\},0,e^{(k)}T_{s,k}/\mu_k\Big),\quad k=1,\ldots,5.
\end{align*} 
Together, these random vectors satisfy the following BAR with some ``nice'' functions $f$.
\begin{equation}
\label{eq:2s5c-bar-appendix}
    -\E[\cala f(X)] = \alpha_1\Big(\E_{e,1}[f(X+\Delta_{e,1})-f(X)]+\sum_{k=1}^5\E_{s,k}[f(X+\Delta_{s,k})-f(X)]\Big),
\end{equation}
where $\E_{e,1}[\cdot]$ and $\E_{s,k}[\cdot]$ are Palm expectations.
For $x=(z,u_1,v)$ and function $f$, the operator $\cala$ is defined as
\begin{align*}
    \cala f(x)=
    -\frac{\partial f}{\partial u_1}(x)&-\frac{\partial f}{\partial v_1}(x)\mathbbm{1}\{z_1>0, z_3=z_5=0\}-\frac{\partial f}{\partial v_3}(x)\mathbbm{1}\{z_3>0, z_5=0\}-\frac{\partial f}{\partial v_5}(x)\\
    &-\frac{\partial f}{\partial v_2}(x)-\frac{\partial f}{\partial v_4}(x)\mathbbm{1}\{z_4>0, z_2=0\}.
\end{align*}

\end{lemma}

Working with the BAR in~\eqref{eq:2s5c-bar-appendix}, we consider the following bounded test functions with $x=(z,u_1,v)$,
\begin{align}
    f_{\theta}(x)&=\exp\Big(\sum_{k=1}^5\theta_kz_k\Big)\exp\Big(-\gamma_1(\theta)(\alpha_1 u_1) + \sum_{k=1}^5\zeta_k(\theta)(\mu_kv_k)\Big).
\end{align}
The functions $\gamma_1(\theta,t)$ and $\zeta(\theta,t)$ chosen to satisfy
\begin{equation}
    e^{\theta_1}\E\Big[e^{-\gamma_1(\theta_1)(T_{e,1})}\Big]=1\quad
    \text{and}\quad e^{-\theta_k+\theta_{k+1}\mathbbm1\{k<5\}}\E\Big[e^{-\zeta_k(\theta)(T_{s,k})}\Big]=1,\quad k\in\{1,\ldots,5\}.\label{eq:general-gamma-def}
\end{equation}
It has been proved in~\cite{BravDaiMiya2017} that such $\gamma$ and $\zeta$ are well defined.
In fact, the $\gamma$ and $\zeta$ notations are intentionally resued, as we shall see their connections to previous definitions in the following Taylor expansions in Lemma~\ref{lem:2s5c-taylor}.

Define
\begin{align}
    &\bar\gamma_1(\theta_1)=\theta_1,\quad\tilde\gamma_1(\theta_1)=\frac{1}{2}c_{e,1}^2\theta_1^2,\label{eq:general-taylor-1}\\
    &\bar\zeta_k(\theta)=\theta_{k+1}\mathbbm1\{k<5\}-\theta_k,\quad\tilde\zeta_k(\theta)=\frac{1}{2}c_{s,k}^2(\theta_{k+1}\mathbbm1\{k<5\}-\theta_k)^2,\quad k\in\{1,\ldots,5\},\\
    &\gamma_1^\ast(\theta_1)=\bar\gamma_1(\theta_1)+\tilde\gamma_1(\theta_1),\quad\zeta_k^\ast(\theta)=\bar\zeta_k(\theta)+\tilde\zeta_k(\theta),\quad k\in\{1,\ldots,5\},\label{eq:general-taylor-3}
\end{align}
where $c_{e,1}^2$ and $c_{s,k}^2$ are defined in~\eqref{eq:ce-def} and~\eqref{eq:cs-def}.
\begin{lemma}[\cite{DaiGlynXu2023} Lemma 5.3]
\label{lem:2s5c-taylor}
Let $\theta(r)\in\R^5$ be given for each $r\in(0,1)$, satisfying $|\theta(r)|\leq cr$. Denoting $\theta=\theta(r)$, as $r\to0$, we have
    \begin{align*}
    &\gamma_1(\theta_1)=\gamma_1^\ast(\theta_1)+o(r^2\theta_1)+o(\theta_1^2),\\
    &\zeta_k(\theta)=\zeta_k^\ast(\theta)+o(r^2|\theta|)+o(|\theta|^2),\quad k\in\{1,\ldots,5\}.
\end{align*}
\end{lemma}

Given these bounded test functions, we define the MGF $\psi^{(r)}(\theta)=\E[f_\theta(X^{(r)})],$
and similarly the conditional MGFs
\begin{align*}
    &\psi^{(r)}_1(\theta)=\E\Big[f_\theta(X^{(r)})|Z^{(r)}_1=0, Z^{(r)}_3=0, Z^{(r)}_5=0\Big],\\
    &\psi^{(r)}_3(\theta)=\E\Big[f_\theta(X^{(r)})| Z^{(r)}_3=0, Z^{(r)}_5=0\Big],\quad
    \psi^{(r)}_5(\theta)=\E\Big[f_\theta(X^{(r)})| Z^{(r)}_5=0\Big],\\
    &\psi_4^{(r)}(\theta)=\E\Big[f_\theta(X^{(r)})| Z_2^{(r)}=0, Z_4^{(r)}=0\Big],\quad
    \psi^{(r)}_2(\theta)=\E\Big[f_\theta(X^{(r)})| Z_2^{(r)}=0\Big].
\end{align*}

In fact, the MGFs $\psi^{(r)}$ for general distributions and MGFs for exponential distributions $\phi^{(r)}$ do not differ, as quantified precisely in the following lemma.
\begin{lemma}[\cite{BravDaiMiya2023} Lemma 2.12]
As $r\to0$,
    \begin{equation*}
        \phi^{(r)}(\theta)-\psi^{(r)}(\theta)=o(1)\quad\text{and}\quad \phi_k^{(r)}(\theta)-\psi_k^{(r)}(\theta)=o(1),\quad k=1,\ldots,5.
    \end{equation*}
\end{lemma}

Following the analysis in \cite{BravDaiMiya2023} and the Taylor expansion, we have the following asymptotic BAR.
\begin{lemma}
For $\theta\in\R_-^5$ satisfying $|\theta(r)|\leq cr$, we have
\begin{equation}
\label{eq:2s5c-asymptotic-bar}
    \begin{aligned}
        &\alpha_1\big(\tilde\gamma_1(\theta_1)+\sum_{k=1}^5\tilde\zeta_k(\theta)\big)\psi(\theta)\\
    &+\beta_1^{(r)}\mu_1^{(r)}\bar\zeta_1(\theta)\Big(\psi(\theta)-\psi_1(\theta)\Big)+\beta_4^{(r)}\mu_4^{(r)}\bar\zeta_4(\theta)\Big(\psi(\theta)-\psi_4(\theta)\Big)\\
    &+\beta_2^{(r)}\big(\mu_2^{(r)}\bar\zeta_2(\theta)-\mu_4^{(r)}\bar\zeta_4(\theta)\big)\big(\psi(\theta)-\psi_2(\theta)\big)\\
    &+\beta_3^{(r)}\big(\mu_3^{(r)}\bar\zeta_3(\theta)-\mu_1^{(r)}\bar\zeta(\theta)\big)\big(\psi(\theta)-\psi_3(\theta)\big)\\
    &+\beta_5^{(r)}\big(\mu_5^{(r)}\bar\zeta_5(\theta)-\mu_3^{(r)}\bar\zeta_3(\theta)\big)\big(\psi(\theta)-\psi_5(\theta)\big)=o(r^2|\theta|) + o(|\theta|^2).
    \end{aligned}
\end{equation}

\end{lemma}

Starting from this asymptotic BAR and applying a similar proof strategy as in Section~\ref{sec:exp-proof}, we can conclude the desired product-form limit under general interarrival and service time distributions. Therefore, we omit the proofs here.

\section{Proof of Uniform Moment Bound}
\label{sec:mom-proof-2s5c}

In this section, we prove the uniform moment bound condition in Proposition~\ref{prop:2s5c-ssc-low}, that $\exists r_0\in(0,1)$ such that
\begin{equation*}
    \sup_{r\in(0,r_0)}\E\Big[\big(rZ_1^{(r)}\big)^{2+\epsilon_0}\Big]<\infty\quad\text{and}\quad\sup_{r\in(0,r_0)}\E\Big[\big(r^2Z_4^{(r)}\big)^{2+\epsilon_0}\Big]<\infty. 
\end{equation*}

To highlight the difference between proving moment bounds in multi-class networks and single-class Jackson networks, we present the proof for exponential distributions and focus on integer-valued moment bounds. Notably, even under exponential distributions, the uniform moment bound is new, making our results valuable even in this simple context.

A key aspect of proving moment bounds for low-priority classes in multi-class networks is leveraging the moment bounds of high-priority classes. \cite{CaoDaiZhan2022} provides a sufficient condition for the existence of moment bounds for high-priority classes, when the MCN under SBP policy has fluid SSC, the high-priority classes of this network has uniform moment bounds.

The proof strategy follows a similar approach used to establish uniform moment bounds in open Jackson networks -- a double induction on the moment order and the station index. For each moment order, we first establish the moment bound for the lowest priority class at Station 1, the lighter-loaded station, and then proceed to Class 4 at Station 2, the lowest-priority class at the heavier-loaded station. We then continue to higher moments in this manner.

The base case, $n=0$, is trivial, allowing us to move directly to the inductive steps. Assume that the uniform moment bound has been established for moments up to $n-1$-th, i.e.,
\begin{equation}
    \sup_{r\in(0,r_0)}\E\Big[\big(rZ_1^{(r)}\big)^m\Big]\leq C_{r,m}\quad\text{and}\quad \sup_{r\in(0,r_0)}\E\Big[\big(r^2Z_4^{(r)}\big)^m\Big]\leq C_{r,m},
\end{equation}
we aim to prove the $n$-th moment bound for $Z_1^{(r)}$. To this end, we consider the following Lyapunov function,
\begin{equation*}
    f_1(z)=\frac{1}{n+1}r^{n-1}\Big(z_1 + \frac{m_3-m_5\frac{m_2}{m_4}}{m_1+m_3-m_5\frac{m_2}{m_4}}z_2 +\frac{m_3}{m_1+m_3-m_5\frac{m_2}{m_4}}z_3 +\frac{m_5}{m_1+m_3-m_5\frac{m_2}{m_4}}z_5\Big)^{n+1}.
\end{equation*}
This expression demonstrates that to prove the moment bound for the lowest-priority class at Station 1, we must also account for high-priority classes at both Station 1 and other stations, the involvement of which does not present an issue, as the uniform moment bounds for these high-priority classes have already been established. 

Applying generator to $f_1$, we have
\begin{align*}
    Gf_1(z)=&r^{n-1}\Big(z_1+ \frac{m_3-m_5\frac{m_2}{m_4}}{m_1+m_3-m_5\frac{m_2}{m_4}}z_2 +\frac{m_3}{m_1+m_3-m_5\frac{m_2}{m_4}}z_3+\frac{m_5}{m_1+m_3-m_5\frac{m_2}{m_4}}z_5\Big)^n\\
    &\Big[\alpha_1-\frac{1}{m_1+m_3-m_5\frac{m_2}{m_4}}\big(\mu_1^{(r)}m_1\mathbbm1\{z_1>0, z_3=z_5=0\}-\mu_2^{(r)}m_5\frac{m_2}{m_4}\mathbbm1\{z_2>0\}\\
    &+\mu_3^{(r)}m_3\mathbbm1\{z_3>0,z_5=0\}+(-\mu_4^{(r)}\mathbbm1\{z_4>0,z_2=0\}+\mu_5^{(r)}\mathbbm1\{z_5>0\})\big)m_5\Big]\\
    &+r^{n-1}\frac{1}{n+1}\sum_{k=2}^n\Bigg[\Big(z_1+ \frac{m_3-m_5\frac{m_2}{m_4}}{m_1+m_3-m_5\frac{m_2}{m_4}}z_2 +\frac{m_3}{m_1+m_3-m_5\frac{m_2}{m_4}}z_3+\frac{m_5}{m_1+m_3-m_5\frac{m_2}{m_4}}z_5\Big)^{n+1-k}\\
    &\Big(\alpha_1-\frac{1}{m_1+m_3-m_5\frac{m_2}{m_4}}\big(\mu_1^{(r)}m_1\mathbbm1\{z_1>0, z_3=z_5=0\}-\mu_2^{(r)}m_5\frac{m_2}{m_4}\mathbbm1\{z_2>0\}\\
    &+\mu_3^{(r)}m_3\mathbbm1\{z_3>0,z_5=0\}+(-\mu_4^{(r)}\mathbbm1\{z_4>0,z_2=0\}+\mu_5^{(r)}\mathbbm1\{z_5>0\})\big)m_5\Big)^k\binom{n+1}{k}\Bigg].
\end{align*}
We note that
\begin{align*}
    &\alpha_1-\frac{1}{m_1+m_3-m_5\frac{m_2}{m_4}}\big(\mu_1^{(r)}m_1\mathbbm1\{z_1>0, z_3=z_5=0\}-\mu_2^{(r)}m_5\frac{m_2}{m_4}\mathbbm1\{z_2>0\}\\
    &+\mu_3^{(r)}m_3\mathbbm1\{z_3>0,z_5=0\}+(-\mu_4^{(r)}\mathbbm1\{z_4>0,z_2=0\}+\mu_5^{(r)}\mathbbm1\{z_5>0\})\big)m_5\\
    &=\alpha_1-\frac{1}{m_1+m_3-m_5\frac{m_2}{m_4}}\Big(\mu_1^{(r)}m_1\mathbbm1\{z_1>0,z_3=z_5=0\}+\mu_3^{(r)}m_3\mathbbm1\{z_3>0,z_5=0\}+\mu_5^{(r)}m_5\mathbbm1\{z_5>0\}\Big)\\
    &+\frac{\frac{m_5}{m_4}}{m_1+m_3-m_5\frac{m_2}{m_4}}\Big(\mu_2^{(r)}m_2\mathbbm1\{z_2>0\}+\mu_4^{(r)}m_4\mathbbm1\{z_4>0,z_2=0\}\Big)\\
    &=\alpha_1-\frac{1}{m_1+m_3-m_5\frac{m_2}{m_4}}\frac{1}{1-r}\Big(1-\mathbbm1\{z_1=z_3=z_5=0\}\Big)+\frac{\frac{m_5}{m_4}}{m_1+m_3-m_5\frac{m_2}{m_4}}\frac{1}{1-r^2}\Big(1-\mathbbm1\{z_4=z_2=0\}\Big)\\
    &\leq \alpha_1-\frac{1}{m_1+m_3-m_5\frac{m_2}{m_4}}\Big(\frac{1}{1-r}-\frac{m_5}{m_4}\frac{1}{1-r^2}\Big)+\frac{1}{m_1+m_3-m_5\frac{m_2}{m_4}}\frac{1}{1-r}\mathbbm1\{z_1=z_3=z_5=0\}\\
    &=\alpha_1-\frac{1-\frac{m_5^{(r)}}{m_4^{(r)}}}{m_1^{(r)}+m_3^{(r)}-m_5^{(r)}\frac{m_2^{(r)}}{m_4^{(r)}}}+\frac{1}{m_1+m_3-m_5\frac{m_2}{m_4}}\frac{1}{1-r}\mathbbm1\{z_1=z_3=z_5=0\}\\
    &\leq-\frac{r}{1-r}+\frac{1}{m_1+m_3-m_5\frac{m_2}{m_4}}\frac{1}{1-r}\mathbbm1\{z_1=z_3=z_5=0\}.
\end{align*}
Therefore, together with expectations, we have
\begin{align*}
    0&=\E\Big[Gf_1(Z^{(r)})\Big]\\
    &\leq -\frac{r^n}{1-r}\E\Big[\Big(Z_1^{(r)} + \frac{m_3-m_5\frac{m_2}{m_4}}{m_1+m_3-m_5\frac{m_2}{m_4}}Z_2^{(r)} +\frac{m_3}{m_1+m_3-m_5\frac{m_2}{m_4}}Z_3^{(r)} +\frac{m_5}{m_1+m_3-m_5\frac{m_2}{m_4}}Z_5^{(r)}\Big)^n\Big]\\
    &+\frac{r^{n-1}}{1-r}\E\Big[\Big(\frac{m_3-m_5\frac{m_2}{m_4}}{m_1+m_3-m_5\frac{m_2}{m_4}}Z_2^{(r)}\Big)^n\Big] + D_{1,n-1},
\end{align*}
where the remaining lower-order terms can be controlled by induction hypothesis. From above inequality, we can conclude the desired moment bound, 
\begin{equation*}
        \E\Big[(rZ_1^{(r)})^n\Big]<\infty.
    \end{equation*}

Now that the $n$-th moment for Class 1 has been proven, we can make use of this bound in the proof of moment bound for Class 4, the low-opriority class at the heavier-loaded Station 2. 
In this case, we consider the following Lyapunov function
\begin{equation*}
    f_2(z)=\frac{1}{n+1}r^{2(n-1)}\Big(\frac{m_2+m_4}{m_4}z_1 + \frac{m_2+m_4}{m_4}z_2 +z_3 +z_4\Big)^{n+1}.
\end{equation*}

Applying generator to $f_2$, we have
\begin{align*}
    Gf_2(z)=&r^{2(n-1)}\Big(\frac{m_2+m_4}{m_4}z_1 + \frac{m_2+m_4}{m_4}z_2 +z_3 +z_4\Big)^n\\
    &\Big[\alpha_1\frac{m_2+m_4}{m_4}-\mu_2^{(r)}\mathbbm1\{z_2>0\}\frac{m_2}{m_4}-\mu_4^{(r)}\mathbbm1\{z_4>0,z_2=0\}\Big]+\text{Residuals}
\end{align*}
It is easy to see that
\begin{align*}
    &\alpha_1\frac{m_2+m_4}{m_4}-\mu_2^{(r)}\mathbbm1\{z_2>0\}\frac{m_2}{m_4}-\mu_4^{(r)}\mathbbm1\{z_4>0,z_2=0\}\\
    &=\frac{1}{m_4}\Big(1-\frac{1}{1-r^2}(\mathbbm1\{z_2>0\}+\mathbbm1\{z_4>0,z_2=0\})\Big)\\
    &=\frac{1}{m_4}\Big(1-\frac{1}{1-r^2}(1-\mathbbm1\{z_4=z_2=0\})\Big)=\frac{1}{m_4}\Big(-\frac{r^2}{1-r^2}+\mathbbm1\{z_4=z_2=0\}\Big).
\end{align*}

Hence, with expectation, we have
\begin{align*}
    0=\E\Big[Gf_2(Z^{(r)}\Big]
    &\leq -\frac{1}{m_4}\frac{r^{2n}}{1-r^2}\E\Big[\Big(\frac{m_2+m_4}{m_4}Z_1^{(r)} + \frac{m_2+m_4}{m_4}Z_2^{(r)} +Z_3^{(r)} +Z_4^{(r)}\Big)^n\Big] \\
    &+ \frac{1}{m_4}r^{2(n-1)}\E\Big[\Big(\frac{m_2+m_4}{m_4}Z_1^{(r)}  +Z_3^{(r)} \Big)^n\mathbbm1\{Z_4^{(r)}=Z_2^{(r)}=0\}\Big] + D_{2,n-1}.
\end{align*}

    To bound the cross term, we first note that when $n\geq2$, it is clear that $2(n-1)\geq n$. As such, in this case, we directly obtain
    \begin{align*}
        &\E\Big[r^{2(n-1)}\Big(\frac{m_2+m_4}{m_2}Z_1^{(r)}+Z_3^{(r)}\Big)^n\cdot\mathbbm{1}\{Z_2^{(r)}=0,Z_4^{(r)}=0\}\Big]\\
        &\leq\E\Big[r^n\Big(\frac{m_2+m_4}{m_2}Z_1^{(r)}+Z_3^{(r)}\Big)^n\cdot\mathbbm{1}\{Z_2^{(r)}=0,Z_4^{(r)}=0\}\Big]\\
        &\leq \E\Big[r^n\Big(\frac{m_2+m_4}{m_2}Z_1^{(r)}+Z_3^{(r)}\Big)^n\Big]<\infty,
    \end{align*}
    for we have proven the desired moment bound for $(rZ_1^{(r)})^n$.
    Otherwise, when $n=1$, we make use of the Cauchy-Schwarz inequality and obtain
    \begin{align*}
        &\E\Big[(\frac{m_2+m_4}{m_2}Z_1^{(r)}+Z_3^{(r)})\cdot\mathbbm{1}\{Z_2^{(r)}=0,Z_4^{(r)}=0\}\Big]\\
        &\leq\sqrt{\E[(\frac{m_2+m_4}{m_2}Z_1^{(r)}+Z_3^{(r)})^2]\E[\mathbbm{1}\{Z_2^{(r)}=0,Z_4^{(r)}=0\}]}\\
        &\leq \frac{m_2+m_4}{m_2}\sqrt{\E[(r(Z_1^{(r)}+Z_3^{(r)}))^2]}.
    \end{align*}
    
Following a similar approach, we have completed the proof 
   \begin{equation*}
       \E_\pi[(r^2Z_4^{(r)})^n]<\infty.
   \end{equation*}

\section{Conclusion}
\label{sec:conclude}

In this companion paper, we have demonstrated the results from \cite{DaiHuo2024} in a specific two-station, five-class reentrant line setup. By focusing on this concrete example, we prove the asymptotic product-form limit and uniform moment bound, highlighting key technical contributions while minimizing complex notation. The numerical experiments, the same as those in~\cite{DaiHuo2024}, confirm the theory's applicability in this context. Our findings underscore the robustness of the product-form limit under multi-scale heavy traffic and lay the groundwork for future exploration of more complex systems and refined performance evaluations. We refer readers who are interested in the general proofs and broader context to the full paper~\cite{DaiHuo2024}.

\bibliography{citation} 

\end{document}